\documentclass[a4paper, reqno, 12pt]{amsart}
\usepackage{tikz, tikz-3dplot}
\usepackage{amsmath, amsthm, amssymb}

\theoremstyle{plain}
\newtheorem{te}{Theorem}[section]
\newtheorem{lem}[te]{Lemma}
\newtheorem{co}[te]{Corollary}
\newtheorem{pr}[te]{Proposition}
\newtheorem{de}[te]{Definition}

\newtheorem{qu}[te]{Question}

\theoremstyle{remark}

\newtheorem*{ack*}{Acknowledgment}

\def\0{{\bf 0}}

\def\R{{\mathbb R}}

\def\N{{\mathbb N}}

\def\S{{\mathbb S}}
\def\Z{{\mathbb Z}}
\def\P{{\mathbb P}}

\def\supp{{\operatorname{supp}\,}}
\def\nint{\mathop{\diagup\kern-13.0pt\int}}

\def\dist{{\operatorname{dist}\,}}

\def\Ic{{\mathcal I}}
\def\Cc{{\mathcal C}}
\def\Tc{{\mathcal T}}

\def\Pc{{\mathcal P}}
\def\Sc{{\mathcal S}}
\def\Lc{{\mathcal L}}

\begin{document}

\title[On the $N$-set occupancy problem ]{On the $N$-set occupancy problem}

\author{Ciprian Demeter}
\address{Department of Mathematics, Indiana University, 831 East 3rd St., Bloomington IN 47405}
\email{demeterc@iu.edu}
\author{Ruixiang Zhang}
\address{Department of Mathematics, UC Berkeley}
\email{ruixiang@berkeley.edu}

\thanks{The first author is  partially supported by the NSF grants  DMS-2055156  and  DMS-2349828. The second  author is supported in part by
	NSF DMS-2143989 and the Sloan Research Fellowship.}

\begin{abstract}
We explore variants of the following open question: Split $[0,1]^2$ into $N^2$ squares with side length $1/N$. Is there a way to select $N$ of these squares such that each line intersects only $O(1)$ of them?

\end{abstract}

\maketitle

\section{introduction}

We begin with an abstract formulation of the problem we will investigate. For an integer $M$ we will write $[M]$ for a set with $M$ elements. The nature of the set will be irrelevant. For $N\le M$, an $N$-set $S$ will be any subset of $[M]$ with $N$ elements, which we write as $|S|=N$.

The value of interest for us will be $M=N^2$.
\begin{de}
\label{de1}	
Given a collection $\Sc$ of $N$-sets in $[N^2]$, we will write
$$C(\Sc)=\min_{S':\;N-set}\max_{S\in \Sc}|S\cap S'|.$$
The minimum is taken over all $N$-sets $S'$ in $[N^2]$.
\end{de}
\smallskip

If $\Sc$ consists of all $N$-sets, then $C(\Sc)=N$. If $|\Sc|\le N-1$ then $C(\Sc)=0$. It is less obvious but not very difficult (see Proposition \ref{prO}) to show that $|\Sc|=O(N)$ implies $C(\Sc)=O(1)$. There are very large collections $\Sc_N$ with $C(\Sc_N)=O(1)$. For example $\Sc_N$ may consist of all the $N$-sets that are disjoint from a fixed $N$-set in $[N^2]$. In this case $|\Sc_N|={{N^2-N}\choose N}$. Another example that is very uniform (each element in $[N^2]$ is contained in roughly $|\Sc_N|/N$ of the sets in $\Sc_N$) consists of all tubes in $\Tc_N$ (see below) whose angle with a horizontal line is at least (say) $\pi/4$.
\medskip
\smallskip

Our motivating example is as follows. Partition $[0,1]^2$ into $N^2$  squares with side length $1/N$. We identify this collection with $[N^2]$, and we let $\Sc$ consist of all collections of squares in $[N^2]$ that intersect a given line segment. We denote this collection by $\Tc_N$, and will refer to each $T\in\Tc_N$ as a ``tube". With this definition, a tube may in fact contain as many as $2N$ squares. We will ignore this issue, as such multiplicative constants are largely irrelevant. Note that there are $\sim N^2$ tubes.
\smallskip

The following elegant question is deceivingly easy-looking, but open.

\begin{qu}
\label{14}	
Is it true that $C(\Tc_N)$ is $O(1)$ as $N\to \infty$? If not, what is its asymptotic growth as $N\to\infty$?	
\end{qu}

The paper \cite{Ca} contains a few results about this question and some of its variants,  and explains their relevance  in Euclidean harmonic analysis. The reader is also encouraged to check the more recent paper \cite{CIW} for applications.

One can ask a similar question when $[0,1]^2$ is replaced with the unit sphere $\S^2$, and tubes consist of unions of $\sim N$ square-like tiles of diameter $1/N$ intersecting great circles. The paper \cite{Li} describes a connection between this problem and the existence of bases of uniformly bounded spherical harmonics on $\S^2$. This connection has been observed earlier by Jean Bourgain, cf. \cite{Va}, who popularized Question \ref{14} as an attractive toy model.
The two geometries -flat vs positively curved- are different, so the questions might in principle have different answers in the two contexts.

In the next section we present essentially sharp examples for families of arbitrary $N$-sets. In Section \ref{sec4} we introduce a strategy aimed  to disprove the $O(1)$ bound for tubes, and an interesting construction that nullifies this attempt. In Section \ref{sec5} we show that no AD-regular N-set is an extremizer for $C(\Tc_N)$.

\begin{ack*}
The  authors would like to thank P\'eter Varj\'u for motivating conversations at the inception of this project.	 They are also indebted to Tony Carbery  for sharing his perspective on the topic of this paper.
\end{ack*}	

\section{Upper bounds for arbitrary collections}
\label{s2}

The  upper bound  $C(\Tc_N)=O(\frac{\log N}{\log\log N})$ is an easy application of the probabilistic method.
It is the case $d=2$, $k\sim \frac{\log N}{\log\log N}$ in Theorem 2 of \cite{Ca}.

Our abstract formulation emphasizes the fact that this method does not use anything special about the fine structure of tubes, other than the polynomial upper bound on their cardinality.
\begin{pr}
\label{18}	
Let $s>0$.	
Consider any collection $\Sc_N$ of $N$-sets $S$ in $[N^2]$, with cardinality  $|\Sc_N|\le N^{1+s}$. Then for each $\beta>1+s$ and each $N$ large enough  we have
$$C(\Sc_N)\le \frac{\beta\log N}{\log \log N}.$$
\end{pr}
\begin{proof}
We use Stirling's formula
$$k!\sim \frac{k^{k+\frac12}}{e^k}$$ with $k=\frac{\beta\log N}{\log \log N}$ to find
$$k!\gtrsim (\frac{\beta}{e\log\log N})^{k}(\log N)^{k}.$$
Then note that for each $\epsilon>0$
$$(C\log\log N)^k\lesssim_\epsilon N^{\epsilon}$$
and also that
$$(\log N)^k=N^{\beta}.$$
Thus $k!\gtrsim_\epsilon N^{\beta-\epsilon}$.

There are $N^2\choose N$ many $N$-sets. Select an $N$-set $S'$ at random, with equal probability. For each $S\in\Sc_N$ and each $0\le k<  N$, we compute the probability
$$\P(|S\cap S'|=k)=\frac{{N\choose k}\times {N^2-N\choose N-k}}{{N^2\choose N}}=\frac{1}{k!}\prod_{i=1}^{k-1}\frac{(N-i)^2}{N^2-i}\prod_{j=0}^{N-1-k}\frac{N^2-N-j}{N^2-k-j}<\frac1{k!}.$$
It is in fact easy to see that
$$\frac{{N\choose k}\times {N^2-N\choose N-k}}{{N^2\choose N}}\sim \frac1{k!},$$
showing no loss in this computation.
Thus
\begin{equation}
\label{23}
\P(|S\cap S'|> k,\;\text{for some } S\in\Sc_N)\le \sum_{S\in\Sc_N}\sum_{l\ge k+1}\P(|S\cap S'|=l)\lesssim\sum_{l\ge k+1}\frac{|\Sc_N|}{l!}\lesssim\frac{N^{1+s}}{k!}.\end{equation}
Using $k=\frac{\beta\log N}{\log\log N}$ we conclude that for large enough $N$
$$\P(|S\cap S'|\le \frac{\beta\log N}{\log\log N},\;\forall S\in\Sc_N)\sim 1.$$
\end{proof}
A slight modification of this argument  shows that if $|\Sc_N|\le N^{O(1)}$ then for each $\epsilon>0$ there is some $N^{1-\epsilon}$-set $S'_{N,\epsilon}$ in $[N^2]$  such that
$$\max_{S\in \Sc_N}|S'_{N,\epsilon}\cap S|=O_\epsilon(1).$$

Since there are $O(N^2)$ tubes, the upper bound $C(\Tc_N)=O(\frac{\log N}{\log\log N})$ is immediate. To the best of our knowledge, there is no improvement of this estimate in the literature.

We next prove that the probabilistic method is essentially  sharp, if only the cardinality of $\Sc_N$ is taken into account. We start with a definition.
\begin{de}
Given $l< k<m$ we let $M(n,k,l)$ be the minimum number of $k$-sets in $[n]$ needed to cover all $l$-sets in $[n]$. By that we mean that each $l$-set needs to be a subset of one of the selected $k$-sets.
\end{de}
Since there are $n\choose l$ many $l$-sets, and since each $k$-set covers exactly $k\choose l$ of them, it follows that $M(n,k,l)\ge {n\choose l}/{k\choose l}$. This lower bound is very close to the best known upper bound.
\begin{te}[\cite{ES}]
\label{17}	
We have
$$M(n,k,l)\lesssim\frac{{n\choose l}}{{k\choose l}}\log{k\choose l}.$$
\end{te}
\begin{proof}
The proof in \cite{ES} is probabilistic. We rewrite it using a double counting argument, to make it more transparent.

Let $x$ be a fixed integer to be chosen at the end of the argument. For (not necessarily distinct) $k$-sets $S_1,S_2,\ldots,S_x$ we write
$$\Lc(S_1,\ldots,S_x)=\{L: l-set\text{ not covered by any }S_i\}.$$
Note that
\begin{align*}
\sum_{S_1,\ldots,S_x:\;k-sets}|\Lc(S_1,\ldots,S_x)|&=\sum_{L:\;l-set}|\{S:\;k-set,\;L\not\subset S\}|^x\\&= {n\choose l}\left[{n\choose k}-{{n-l}\choose{k-l}}\right]^x={n\choose l}{n\choose k}^x\left[1-\frac{{k\choose l}}{{n\choose l}}\right]^x.
\end{align*}
It follows that there are $S_1,\ldots,S_x$ with
$$|\Lc(S_1,\ldots,S_x)|\le {n\choose l}\left[1-\frac{{k\choose l}}{{n\choose l}}\right]^x.$$
For each $L\in \Lc(S_1,\ldots,S_x)$, pick some $k$-set $S(L)$ covering it. It follows that the collection consisting of $S_1,\ldots,S_x$ together with $\{S(L):\;L\in \Lc(S_1,\ldots,S_x) \}$ cover all $l$-sets. Thus
$$M(m,k,l)\le {n\choose l}\left[1-\frac{{k\choose l}}{{n\choose l}}\right]^x+x.$$
Picking $x\sim \frac{{n\choose l}}{{k\choose l}}\log{k\choose l}$ finishes the argument.

\end{proof}
Armed with this result we can produce small and rather uniform collections of $N$-sets with large intersective constants.
\begin{te}
\label{19}	
Given $m\lesssim  N^{1/3}$, there is a collection $\Sc_m$ of $N$-sets in $[N^2] $ with cardinality $$m^{m-1}N\lesssim |\Sc_m|\lesssim m^{m}N\log N$$ such that
$$C(\Sc_m)\ge m.$$
Moreover, each $n\in[N^2]$ belongs to at least $\sim m^{m-1}$ sets in $\Sc_m$.
\end{te}
\begin{proof}
We first split $[N^2]$ into $N/m$ sets $S_i$, each having cardinality $Nm$. Using Theorem \ref{17}, the $m$-sets in each $S_i$ can be covered with a collection  $\Cc_i$ consisting of
$$\frac{{Nm\choose m}}{{N\choose m}}\le M(Nm,N,m)\lesssim \frac{{Nm\choose m}}{{N\choose m}}\log{N\choose m}$$
$N$-sets. Let $\Sc_m=\cup_{i=1}^{N/m}\Cc_i$. Next we estimate
\begin{align*}
m^m\le \frac{{Nm\choose m}}{{N\choose m}}=\frac{Nm}{N}\frac{Nm-1}{N-1}\ldots\frac{Nm-m+1}{N-m+1}&\le (\frac{Nm-m}{N-m})^m\\&=m^m(1+\frac{m^2-m}{N-m})^m\\&\le m^m(1+\frac{2m^2}{N})^{\frac{N}{2m^2}\frac{2m^3}{N}}\lesssim m^{m}.
\end{align*}
Note also that  ${N\choose m}\le  N^m$. The desired  bounds on $|\Sc_m|$ follow.

To see that $C(\Sc_m)\ge m$, let $S'$ be an arbitrary $N$-set in $[N^2]$. Then $S'$ must contain an $m$-set in some $S_i$. This is covered by some $S\in\Cc_i$, showing that $|S'\cap S|\ge m$.

Finally, fix an element $n$ of $[N^2]$, say in $S_i$. Since the $N$-sets $S$ in $\Cc_i$ containing $n$ cover all $m$-sets in $S_i$ containing $n$, the $(N-1)$-sets $S\setminus \{n\}$ (with $S$ containing $n$) must cover all $(m-1)$-sets in $S_i\setminus \{n\} $. It follows that there are at least
$$M(Nm-1,N-1,m-1)\ge \frac{{Nm-1\choose m-1}}{{N-1\choose m-1}}\sim m^{m-1}$$
sets $S\in \Sc_m$ containing $n$. This is sharp, up to a factor of $m\log N$. Indeed, double counting shows that the average value over $n\in[N^2]$ for the number of $S\in \Sc_m$ containing $n$ is $O(m^m\log N)$. This shows that the collection $\Sc_m$ is rather uniform.

\end{proof}
Let us test the result from the previous theorem with $m=\frac{s\log N}{\log\log N}$, for some fixed $s>0$. We find that
$${N\log N} m^m=N^{1+s}\log N(\frac{s}{\log\log N})^{\frac{s\log N}{\log\log N}}\ll N^{1+s}.$$
The following corollary should be compared with Proposition \ref{18}.
\begin{co}
For each $s$ and each $N$ large enough, there is a collection $\Sc_N$ of $N$-sets in $[N^2]$ such that $|\Sc_N|\ll N^{1+s}$ and $C(\Sc_N)\ge \frac{s\log N}{\log\log N}$.
\end{co}
It is worth comparing the sets in $\Sc_N$ with the collection of tubes in $\Tc_N$. The  construction of $\Sc_N$ in the proof of Theorem \ref{17} is probabilistic, but there is an instructive deterministic construction with only slightly larger cardinality $|\Sc_N|=O(N^{1+s(1+\frac1{\log\log N})})$. Namely, let us further split the sets $S_i$ from the proof of Theorem \ref{19} into $m^2$ sets $S_{i,j}$ having cardinality $N/m$. We may then let $\Cc_i$ consist of all unions of exactly $m$ of the sets $S_{i,j}$.
It is immediate that these $N$-sets cover all $m$-sets in $S_i$.
Letting as before $\Sc_m=\bigcup_{i=1}^{N/m}\Cc_i$ we have $C(\Sc_m)\ge m$. We also have $|\Sc_m|=\frac{N}{m}{m^2\choose m}$. When $m=
\frac{s\log N}{\log\log N}$ we find that  $|\Sc_m|=O(N^{1+s(1+\frac1{\log\log N})})$.

Let us select each $S_{i,j}$ to be a horizontal stack of consecutive squares. Then each tube $T\in\Tc_N$ that fits inside $S_i$ can be essentially assembled (if we allow negligible perturbations of tubes) out of such $S_{i,j}$. This is an easy geometric exercise that uses eccentricity. Thus, $\Sc_m$ will essentially contain all tubes in $\Tc_N$ whose angle with the horizontal axis is  $O(m/N)$. However, $\Sc_m$ will contain other highly disconnected sets, that only resemble tubes at smaller scales.
\medskip

The probabilistic method described earlier is performing poorly for collections of $N$-sets with small cardinality $\lesssim_\epsilon N^{1+\epsilon}$, due to the crude use of the triangle inequality in \eqref{23}. We next present an alternative argument that gives the sharp bound for $C(\Sc)$ when $|\Sc|=O(N)$.
\begin{pr}
\label{prO}	
Assume $K\lesssim N$.	
For each collection $\Sc$ of $N$-sets in $[N^2]$ with $|\Sc|\le KN$ we have
$$C(\Sc)\lesssim K,$$
with a universal implicit constant independent of $K,N,\Sc$.

In particular, we have that $C(\Sc)=O(1)$ if $|\Sc|=O(N)$.
\end{pr}
\begin{proof}
Note first that
$$|\{n\in[N^2]:\sum_{S\in\Sc}1_S(n)>2K\}|\le \frac{|\Sc|N}{2K}\le \frac{N^2}{2}.$$
Thus the pool
$$\Pc=\{n\in[N^2]:\;\sum_{S\in\Sc}1_S(n)\le 2K\}$$
has cardinality at least $N^2/2$.

We next describe the first stage of a multi-stage selection process for the optimal $S'$. The first stage involves many steps, and the remaining stages are repetitions of the first one.
\\
\\
Step 1. Pick any $n_1\in \Pc$, put $n_1$ in $S'$. Let
$$\Ic(n_1) =\bigcup_{S\in\Sc:\atop{n_1\in S}}S$$
and note that
$$|\Ic(n_1)|\le 2KN.$$
\\
\\
Step 2. If $N^2/2>2KN$, pick any  $n_2\in\Pc\setminus\Ic(n_1)$ and add $n_2$ to $S'$.
\\
\\
Step 3. If $N^2/2>4KN$, pick any $n_3\in\Pc\setminus (\Ic(n_1)\cup \Ic(n_2))$ and add it to $S'$.
\\
\\
We stop when there is no element left to select. At this point we have selected at least $\frac{\frac{N^2}{2}}{2KN}$ elements $n_i$. If we have selected $N$ elements, we stop (at the step when this was achieved) and call $S'$ the set of these $N$ elements.
\\
\\
Otherwise, we call it $S_1'$ and note that
$$\frac{N}{4K}\le |S_1'|<N.$$ We repeat the selection algorithm from the first stage, this time working with the updated pool $\Pc\setminus S_1'$ having cardinality
$$|\Pc\setminus S_1'|\ge \frac{N^2}2-N\ge \frac{N^2}{4}.$$
At the end of this second stage, we produce a set $S_2'$ disjoint from $S_1'$, with cardinality
$$\frac{N}{8K}\le |S_2'|<N.$$
If $|S_1'\cup S_2'|>N$, we stop, and choose $S'$ to be any $N$-set in $S_1'\cup S_2'$. Otherwise, we repeat the algorithm from the first stage for the updated pool with cardinality
$$|\Pc\setminus(S_1'\cup S_2')|\ge \frac{N^2}{2}-2N\ge \frac{N^2}{4}.$$

At the end of this process, we find pairwise disjoint  sets $S_1',\ldots, S_k'$,  with $k\le 8K\ll N$, each of which has size at least $N/8K$. The set $S'$ can be chosen to be any $N$-set in $S_1'\cup\ldots\cup S_k'$.

We note that each $S\in \Sc$ contains at most one $n$ from each $S_i'$. Let us see this for $S_1'$, as the argument for the other $S_i'$ is identical. For $S\in\Sc$, let $n_j$ be the first selected element with $n_j\in S$. Recall that all elements selected later are in fact chosen from a pool that has $\Ic(n_j)$, and thus also $S$ in particular,  removed. This means that $S$ may not contain any $n_{j'}$ with $j'>j$.

Since there are $O(K)$ stages, each $S\in\Sc$ will intersect $S'$ at most $O(K)$ times, showing that $C(\Sc)=O(K)$.

\end{proof}

\section{A few examples for tubes}
\label{s3}

There are many explicit examples that illustrate the fact that $C(\Tc_N)\lesssim \sqrt{N}$. One way to select the set $S'$ in  Definition \ref{de1} is by selecting exactly one $1/N$-square inside each $1/\sqrt{N}$-square. Another one is to pack all the $1/N$-squares inside a fixed $1/\sqrt{N}$-square. Or to place them on a curve having positive curvature. More complicated constructions will arise naturally in the next section.
\smallskip

Here is a self-similar construction that illustrates the upper bound $C(\Tc_N)\lesssim N^\epsilon$. It is easy to see that $C(\Tc_N)\le C(\Tc_{\sqrt{N}})^2$. Indeed, let $S'$ be the optimal configuration of $1/\sqrt{N}$-squares for $C(\Tc_{\sqrt{N}})$. Inside each of  these squares we place the rescaled version of $S'$, consisting of $C(\Tc_{\sqrt{N}})$ many $1/N$-squares. We have a total of $C(\Tc_{\sqrt{N}})^2$ many $1/N$-squares, call them $S''$. Since each $T\in \Tc_N$ sits inside some $T'\in\Tc_{\sqrt{N}}$, the inequality $|S''\cap T|\le C(\Tc_{\sqrt{N}})^2$ follows.

Thus $C(\Tc_{2^{2^{n}}})\le C(\Tc_{2^{2^{m}}})^{2^{n-m}}$ for $n>m$. Fix $\epsilon>0$.
Since we know from Proposition \ref{18} that $C(\Tc_N)\lesssim_\epsilon N^\epsilon$, we may pick $m$ large enough so that $C(\Tc_{2^{2^{m}}})\le 2^{\epsilon2^{m}}$. Let $S'$ be a set that realizes this bound. We rescale it and populate larger squares using the procedure described above. For each $N=2^{2^n}$ with $n>m$ we get a favorable  self-similar set that realizes the bound $C(\Tc_N)\le N^{\epsilon}$.

\section{A failed attempt to disprove the upper bound $C(\Tc_N)\lesssim 1$ for tubes;
	more examples}
\label{sec4}

We reformulate (essentially equivalently) the main question for tubes as follows. Let $\pi$ be a permutation of $\{1,2,\ldots,N\}$. We identify each square with the point $(n/N,\pi(n)/N)$, for some $1\le n\le N$. Call $\Pc$ the collection of all these points. The $O(1)$ neighborhood $T_{\alpha,\beta}$ of the line $y=\alpha x-\beta$ contains the point $(n/N,\pi(n)/N)$ if and only if
$$|\alpha\frac{n}{N}-\frac{\pi(n)}{N}-\beta|\lesssim\frac{\max\{|\alpha|,1\}}{N}.$$
For  $1\le n\not=m\le N$, call
$$c_{n,m}=\max\{1,|\frac{\pi(n)-\pi(m)}{n-m}|\}$$
$$I_{n,m}=[\frac{\pi(n)-\pi(m)}{n-m}-\frac{O(c_{n,m})}{|n-m|},\frac{\pi(n)-\pi(m)}{n-m}+\frac{O(c_{n,m})}{|n-m|}].$$
\begin{lem}We have
$$\sup_{\alpha,\beta\in\R}|T_{\alpha,\beta}\cap \Pc|\sim\max_n\|\sum_{m\not=n}1_{I_{n,m}}\|_{L^\infty}.$$
\end{lem}	
\begin{proof}
We prove the double inequality. We start with $\lesssim$. Fix $\alpha,\beta$ and fix some $(n/N,\pi(n)/N)\in T_{\alpha,\beta}\cap \Pc$. Pick any $(m/N,\pi(m)/N)\in T_{\alpha,\beta}\cap \Pc$ with $|m-n|\gg 1$; if no such $m$ exists we are done. It follows that
$$|\alpha(n-m)-(\pi(n)-\pi(m))|\lesssim \max\{|\alpha|,1\},$$
or
$$|\alpha-\frac{\pi(n)-\pi(m)}{n-m}|\lesssim \frac{1}{|n-m|}\max\{|\alpha|,1\}.$$
Let us prove that this implies $\alpha\in I_{n,m}$. It suffices to prove that
$$\max\{|\alpha|,1\}\lesssim \max\{|\frac{\pi(n)-\pi(m)}{n-m}|,1\}.$$
We only need to prove this when $|\alpha|\ge 1$. Recall that we have
$$|\alpha-\frac{\pi(n)-\pi(m)}{n-m}|\le \frac{1}{2}\max\{|\alpha|,1\}.$$
This forces $|\alpha|\sim |\frac{\pi(n)-\pi(m)}{n-m}|$, so we are done.

We next prove $\gtrsim$. Fix $n$ and consider $|m-n|\gg 1$ such that $\alpha\in I_{n,m}$. We have
$$|\alpha-\frac{\pi(n)-\pi(m)}{n-m}|\lesssim\frac{1}{|n-m|}\max\{1,|\frac{\pi(n)-\pi(m)}{n-m}|\},$$
or, letting $\beta=\frac{n\alpha-\pi(n)}{N}$
$$|\alpha\frac{m}{N}-\frac{\pi(m)}{N}-\beta|\lesssim\frac1{N}\max\{1,|\frac{\pi(n)-\pi(m)}{n-m}|\}.$$
Since
$$|\alpha-\frac{\pi(n)-\pi(m)}{n-m}|\le\frac{1}{2}\max\{1,|\frac{\pi(n)-\pi(m)}{n-m}|\},$$
we find as before that
$$\frac1{N}\max\{1,|\frac{\pi(n)-\pi(m)}{n-m}|\}\lesssim \frac1{N}\max\{1,|\alpha|\}.$$
This shows that $(m/N,\pi(m)/N)\in T_{\alpha,\beta}$. The question on whether $C(\Tc_N)=O(1)$ is essentially equivalent to the following one.

\end{proof}
\begin{qu}
\label{9}	
Is there  a permutation $\pi=\pi_N$ with $$\max_{1\le n\le N}\|\sum_{m\not=n}1_{I_{n,m}}\|_{L^\infty}=O(1)?$$	
\end{qu}	
Let us derive some consequences, assuming the answer is yes for such $\pi=\pi_N$. Define
$$A_0(n)=\{m\not=n:\;|\frac{\pi(n)-\pi(m)}{n-m}|\le 1\},$$
and for $1\le j\lesssim \log N$
$$A_j(n)=\{m\not=n:\;|\frac{\pi(n)-\pi(m)}{n-m}|\sim 2^j\}.$$
We will use the inequality
$$\|f\|_1\le |\supp f|\|f\|_{\infty}.$$
For each $n$, the intervals $I_{n,m}$ with $m\in A_0(n)$ lie inside $[-O(1),O(1)]$.
Our assumption then implies that for each $n$ we have
$$\sum_{m\in A_0(n)}|I_{n,m}|\lesssim 1,$$
or
\begin{equation}
\label{1}
\sum_{m\in A_0(n)}\frac{1}{|n-m|}\lesssim 1.
\end{equation}
Similarly, since $I_{n,m}\subset [-O(2^j),O(2^j)]$ for $j\ge 1$, we should also have
\begin{equation}
\label{2}
\sum_{m\in A_j(n)}\frac1{|n-m|}\lesssim 1.
\end{equation}
Let us first digest of difficulty of simultaneously having \eqref{1} and \eqref{2}.
\begin{pr}\label{Lippr}
Assume $\pi$ is Lipschitz
$$\max_{n\not=m}|\frac{\pi(n)-\pi(m)}{n-m}|=O(1).$$Then there exists $n$ and $j$ such that
$$\sum_{m\in A_j(n)}\frac1{|n-m|}\gtrsim \log N.$$
Thus, there is a tube containing at least $\log N$ points.
\end{pr}
\begin{proof}
Only finitely many $j$ may contribute. The result follows since
$$\sum_{1\le n\not=m\le N}\frac{1}{|n-m|}\sim N\log N.$$
\end{proof}

In fact, it is easy to show that for each set $S\subset \{1,2,\ldots,N\}$ with $|S|\sim N$ we have
\begin{equation}
\label{12}
\sum_{n\not=m\in S}\frac1{|n-m|}\sim N\log N.\end{equation}
Thus, since there are $\sim \log N$ values of $j$, any permutation satisfying \eqref{1} and \eqref{2} would have to be very ``uniform".

Let us color the edges of the complete graph $K_N$ with vertices $\{1,2,\ldots,N\}$ using red and blue. We use red for the edge $(n,m)$ between $n$ and $m$ if $m\in\cup_{j\ge 1}A_j(n)$. We use blue otherwise, namely when $m\in A_0(n)$. If \eqref{1} and \eqref{2} hold then
$$\sum_{(n,m)\text{ is blue}}\frac1{|n-m|}\lesssim N$$
$$\sum_{(n,m)\text{ is red}}\frac1{|\pi(n)-\pi(m)|}\lesssim N.$$
If each two-coloring of $K_N$ was forced to contain a monochromatic complete graph $K_M$ with $M\sim N$, this would lead to an immediate contradiction, due to \eqref{12}. This is however far from being true. An old lower bound of Erd\"os on Ramsey numbers produces a two-coloring with the largest monochromatic $K_M$ of order $M\sim \log N$.
\medskip

We next show that  \eqref{1} and \eqref{2} can in fact be realized.

\begin{te}
\label{4.4}	
Let $f:\N\to[0,1]$ satisfy
\begin{equation}
\label{3}
|f(n)-f(m)|\gtrsim \frac1{|n-m|}
\end{equation}
for each $1\le n\not=m$.
Then the function $\pi=\pi_N:\{1,2,\ldots,N\}\to [0,N]$ given by
$$\pi(n)=Nf(n)$$
is a quasi-permutation, in the sense that $|\pi(n)-\pi(m)|\gtrsim 1$ whenever $n\not=m$.

Moreover, it satisfies \eqref{1} and \eqref{2}.	
\end{te}
The fact that $\pi$ does not take integer values is irrelevant for our line of investigation. The requirement \eqref{3} seems very strong, forcing the bounded $f$ to be very ``structured". One example of such $f$ we are aware of is
$$f(n)=\{n\theta\},$$
where $\{x\}=x-\lfloor x\rfloor$ is the fractional part of $x$, and $\theta$ is any badly approximable number. That means
$$\inf_{n}n\;\dist(n\theta,\Z)>0.$$
Examples of such $\theta$ include all quadratic irrationals. To check \eqref{3} note that
$$|f(n)-f(m)|=|(n-m)\theta-(\lfloor n\theta\rfloor-\lfloor m\theta\rfloor)|\gtrsim\frac1{|n-m|}.$$	
\begin{proof}(of Theorem \ref{4.4})
Let us first check \eqref{2}. For $k\ge 0$ define
$$A_n(j,k)=\{m\in A_n(j):\;|n-m|\sim 2^k\}.$$
We collect a few observations. First, assume $A_n(j,k)$ is nonempty. Then, picking some $m$ in it implies that
$$\frac{N}{|m-n|}\lesssim |\pi(n)-\pi(m)|\sim |n-m|2^j,$$
so
\begin{equation}
\label{4}
N\lesssim 2^{2k+j}.
\end{equation}
Second, since $N\gtrsim |\pi(n)-\pi(m)|\sim |n-m|2^j$, we also have
\begin{equation}
\label{5}
N\gtrsim 2^{k+j}.
\end{equation}

For the third observation assume $m\in A_n(j,k)$ satisfies, say,  $m\in J:=[n+2^k,n+2^{k+1}]$. Let $\gamma=\frac{2^{j+k}}{N}$. We know that $\gamma\lesssim 1$.
 We split $J$ into roughly  $\sim 2^k\gamma$ intervals $I$ with length $\ll 1/\gamma$,  a bit smaller than $1/\gamma$. From \eqref{4} we know that there is at least one such interval. We claim that for each such $I$ we have
 \begin{equation}
 \label{6}
 |I\cap A_n(j,k)|\le 1,
 \end{equation}
 so that
 \begin{equation}
 \label{7}
 |A_n(j,k)|\le \frac{2^{2k+j}}{N}.
 \end{equation}
Indeed, assume for contradiction that the intersection contains two distinct $m_1,m_2$. Then
$$|\pi(m_1)-\pi(m_2)|\gtrsim \frac{N}{|m_1-m_2|}\gg N\gamma=2^{j+k}.$$
This however contradicts the fact that $|\pi(m_1)-\pi(n)|,|\pi(m_2)-\pi(n)|\sim 2^{j+k}$, proving that  \eqref{6} must hold.

The verification of \eqref{2} is now immediate as follows
$$\sum_{m\in A_n(j)}\frac1{|n-m|}=\sum_{1\le 2^k\lesssim N2^{-j}}\sum_{m\in A_n(j,k)}\frac1{2^k}\lesssim \sum_{1\le 2^k\lesssim N2^{-j}}\frac{2^{k+j}}{N}\lesssim 1.$$

The argument for \eqref{1} is very similar. For $1\le 2^k,2^l\le N$ we let
$$B_n(k,l)=\{m:\;|m-n|\sim 2^k,\;|\pi(m)-\pi(n)|\sim 2^l\}.$$
Reasoning as before we find that $|B_{n}(k,l)|\lesssim 2^k/\Delta$, where $\Delta=N/2^l\lesssim 2^k$. Thus we have
$|B_n(k,l)|\lesssim 2^{k+l}/N$ and
$$\sum_{m\in A_n(0)}\frac1{|n-m|}\lesssim \sum_{2^k\le N}\sum_{l\le k}\frac1{2^k}|B_n(k,l)|\lesssim \sum_{2^k\le N}\frac{2^k}{N}\lesssim 1.$$

\end{proof}	
Unfortunately, the example $\pi(n)=N\{n\theta\}$ is very bad for Question \ref{9}, for all irrational $\theta\in[0,1]$. By Diriclet's Theorem we may pick $|p|\le \sqrt{N}$ such that $\{p\theta\}\le 1/\sqrt{N}$. It follows that $\pi(pl)=l\pi(p)$ for $l\le \sqrt{N}$. The points $(pl,\pi(pl))$ are in fact collinear, showing that
$$\max_{1\le n\le N}\|\sum_{m\not=n}1_{I_{n,m}}\|_{L^\infty}\gtrsim \sqrt{N}.$$
There is another class of permutations that perhaps deserve some further exploration. It is defined as follows. Write each $1\le n\le N=2^t$ in base 2
$$n=a_0+a_12+\ldots+a_{t-1}2^{t-1},\;a_i\in\{0,1\}.$$
Define
$$\pi(n)=a_{t-1}+a_{t-2}2+\ldots+a_02^{t-1}.$$
This permutation almost satisfies (in a way we will not try to quantify precisely) $$|\pi(n)-\pi(m)||n-m|\gtrsim N.$$ Unfortunately, this is as bad for Question \ref{9} as the previous example. All base-two palindromes (numbers $n$ whose ordered sequence of digits in base 2 is the same when read  backwards) satisfy $\pi(n)=n$. As there are roughly $\sqrt{N}$ palindromes, we get the same lower bound as before. It would be interesting to explore other similar examples, such as the random scrambling of the digits, as opposed to reversing the digits.

\section{A lower bound for $\max_{T\in \Tc_N}|T\cap S'|$ when $S'$ is an ``AD-regular" set}
\label{sec5}

In this section we generalize the argument from Proposition \ref{Lippr} to reach a stronger conclusion.
\begin{pr}\label{ADregcounter}
    Assume $S'$ is a collection consisting of $N$ squares of side length $\frac{1}{N}$ inside $[0, 1]^2$ that satisfies the following condition: the $r$-neighborhood of these squares can be covered by $\lesssim \frac{1}{r}$ balls of radius $r$, $\forall r \in [\frac{1}{N}, 1]$. Then $\max_{T\in \Tc_N}|T\cap S'| \gtrsim \log N$.
\end{pr}
In particular, the above condition is satisfied if $S'$ is ``AD-regular'' in the following sense: the $r$-neighborhood of every square in $S'$ contains $\sim rN$ squares in $S'$, $\forall r \in [\frac{1}{N}, 1]$. Thus, Proposition \ref{ADregcounter} and the fact that $C(\Tc_N)\lesssim \frac{\log N}{\log\log N} $ tells us that AD-regular sets  are not extremizers for $C(\Tc_N)$. Such an extremizer $S'$ should be more ``spread out''. Perhaps a good candidate to consider is a set with exactly one square inside each $\frac{1}{\sqrt{N}}$-square, as the energy \eqref{avgsum} in the proof below is $\sim N^2$ in that case. These considerations also apply to the $\S^2$ version of the problem.

\begin{proof}[Proof of Proposition \ref{ADregcounter}]
    For $q_1, q_2 \in S'$, define  $\text{dist} (q_1, q_2)$ to be the distance between their centers. Note that for each pair $(q_1, q_2)$, the number of tubes in $\Tc_N$ containing both of them is $\sim \frac{1}{\text{dist} (q_1, q_2)}$.

    For each $q_1 \in S'$, define $$R(q_1) := \sum_{q_2\neq q_1 \in S'} \# \{\text{tubes in } \Tc_N \text{ containing both } q_1 \text{ and } q_2\} \sim \sum_{q_2\neq q_1 \in S'} \frac{1}{\text{dist} (q_1, q_2)}.$$
    We now consider the sum
    \begin{equation}\label{avgsum}
        \sum_{q_1 \in S'} R(q_1) \sim  \sum_{q_1\neq q_2, q_1, q_2 \in S'} \frac{1}{\text{dist} (q_1, q_2)}.
    \end{equation}
    By our hypothesis, for each dyadic number $2^j \in [\frac{1}{N}, 1]$, $S'$ can be covered by $\lesssim 2^{-j}$ balls of radius $2^j$. By Cauchy--Schwarz, the number of $(q_1, q_2) \in S' \times S'$ such that $\text{dist} (q_1, q_2) \leq 2^j$ is $\gtrsim (2^j N)^2 \cdot 2^{-j} = 2^j N^2.$ Using this in \eqref{avgsum} we obtain $\sum_{q_1 \in S'} R(q_1) \gtrsim N^2 \log N$. Hence we can find some $q_1 \in S'$ such that $R(q_1) \gtrsim N \log N$. Unpacking  the definition of $R(q_1)$, we conclude that there is a tube in $\Tc_N$ that contains $q_1$ and $\gtrsim \log N$ other squares in $S'$.

\end{proof}

\end{document}